\documentclass[11pt]{amsart}

\usepackage{amssymb}
\usepackage{url}
\usepackage{amscd}
\usepackage{texdraw}
\usepackage{color}

\theoremstyle{plain}
\newtheorem{thm}{Theorem}

\newtheorem{lem}{Lemma}
\newtheorem{rem}{Remark}
\newtheorem{pro}{Proposition}
\newtheorem{defi}{Definition}

\input txdtools

\begin{document}

\title[Geometrization]
{Geometrization of sub-hyperbolic semi-rational branched coverings}
\author{Tao Cheng and Yunping Jiang}

\address{Tao Cheng: Department Of Mathematics,
East China Normal University,
Shanghai, 200241,
People'S Republic Of China}
\email{tcheng@math.ecnu.edu.cn}

\address{Yunping Jiang: Department Of Mathematics,
Queens College Of Cuny,
Flushing, New York 11367, Usa
And
Department Of Mathematics,
Cuny Graduate Center,
New York, New York 10016, Usa}
\email{yunping.jiang@qc.cuny.edu}

\thanks{The first author is partially supported by National Natural Science Foundation of China (No.11001081 and No.11071074) and the second author is partially supported by
a Simons Collaboration Grant for Mathematics and by a grant from Academy of Mathematics and Systems Science and from the Morningside Center of Mathematics at the Chinese Academy of Sciences and by PSC-CUNY awards.}


\begin{abstract}
Given a sub-hyperbolic semi-rational branched covering which is not CLH-equivalent a rational map,
it must have the non-empty canonical Thurston obstruction.
By using this canonical Thurston obstruction, we decompose this dynamical system in this paper into several sub-dynamical systems. Each of these sub-dynamical systems
is either a post-critically finite type branched covering or a sub-hyperbolic semi-rational type branched covering.
If a sub-dynamical system is a post-critically finite type branched covering with a hyperbolic orbifold, then it has no Thurston obstruction
and is combinatorially equivalent to a unique post-critically finite rational map (up to conjugation by an automorphism of the Riemann sphere)
and, more importantly, if a sub-dynamical system is a sub-hyperbolic semi-rational type branched covering with hyperbolic orbifold, we prove in this paper that it has no Thurston obstruction
and is CLH-equivalent to a unique geometrically finite rational map (up to conjugation by an automorphism of the Riemann sphere).
\end{abstract}

\maketitle

\section{Introduction}

It has been an interesting and important problem for the last thirty years
to discover a necessary and sufficiently condition for a branched covering
of the two-sphere to be realized by a rational map.

Thurston gave a first necessary and sufficiently condition such that a post-critically finite
branched covering (i.e., the set of post-critical orbits is finite) can be realized by a rational map under combinatorial equivalence.
A complete proof of Thurston's result was written by Douady and Hubbard in~\cite{DH}.
In this discovery, Thurston defined a topological obstruction which is called a Thurston obstruction now.
Then a post-critically finite branched covering with a hyperbolic orbifold is combinatorial equivalent to a rational map if and only if
it has no Thurston obstruction. Moreover, the realized rational map is unique up to conjugation
by an automorphism of the Riemann sphere. Pilgrim~\cite{Pi1} further proved that if $f$ with hyperbolic orbifold has a Thurston obstruction, then it must have the canonical one.
Based on this canonical Thurston obstruction, the two-sphere now can be decomposed into a collection of spheres with
holes. After replacing every hole with a disk marked by a point, one can obtain a collection of topological two-spheres and
a new branched covering on each periodic two-sphere. Pilgrim~\cite{Pi2} asked the following question: Does this new branched
covering if it has a hyperbolic orbifold has a Thurston obstruction? Recently, Selinger~\cite{Se} proved that this new branched
covering if it has a hyperbolic orbifold, then it has no Thurston obstruction and thus is combinatorially equivalent to a rational map.
Bonnot and Yampolsky~\cite{BY} soon proposed another approach to prove this new result by applying Minsky's theorem
in~\cite{Mi} and Haissinsky's theorem in~\cite{Ha}.

Things became very different when one turns to geometrically finite branched coverings.
In~\cite{CJ} (see also~\cite{CJS}), we have showed that there exists a geometrically finite branched
covering such that it has no Thurston obstruction and it is not
combinatorially equivalent to any rational map. Based on this counter-example,
an important type of a branched covering named as a sub-hyperbolic
semi-rational branched covering is considered in~\cite{CJ}, ~\cite{CJS}. Furthermore, a new
combinatorial equivalence called the CLH equivalence (combinatorial and
locally holomorphic equivalence) is introduced in~\cite{CJ}, ~\cite{CJS}.
Furthermore, in~\cite{CJS} (see also in~\cite{CT}) and in~\cite{ZJ},
by using two completely different methods, we have proved that a sub-hyperbolic semi-rational branched covering is CLH equivalent to
a rational map if and only if $f$ has no Thurston obstruction. Moreover, the realized rational map is unique up to conjugation by an automorphism of the Riemann sphere.
We would like to note that the orbifold associated to a sub-hyperbolic semi-rational branched covering is always hyperbolic due to the set of post-critical orbits is infinite. Recently, we proved in~\cite{ChJ} that if a sub-hyperbolic semi-rational branched covering has a Thurston obstruction, then it must have the canonical one.

Similarly, based on this canonical Thurston obstruction, the two-sphere now can be decomposed into a collection of spheres with
holes. After replacing every hole with a disk marked by a point, one can obtain a collection of topological two-spheres, and
a collection of new branched coverings on periodic two-spheres. These new branched coverings can be either post-critically finite type
or sub-hyperbolic semi-rational type branched covering. In this paper, we would like to study the problem
that new branched coverings have Thurston obstructions or not?

{\em Our main result (Theorem~\ref{mainth}) says that these new branched coverings have no Thurston obstruction.
More precisely, if a new branched covering is a post-critical finite type branched covering with a hyperbolic orbifold,
then it has no Thurston obstruction and is combinatorially equivalent to a unique post-critically finite rational map
(up to conjugation by an automorphism of the Riemann sphere) and if a new branched covering is a sub-hyperbolic semi-rational
type branched covering with hyperbolic orbifold, then it has no Thurston obstruction and is CLH-equivalent to
a unique geometrically finite rational map (up to conjugation by an automorphism of the Riemann sphere).}

In the post-critically finite case, one works on a finite-dimensional Teichm\"uller space
of the Riemann sphere minus the set of post-critical orbits. Therefore, one can use
Minsky's theorem in~\cite{Mi}, Haissinsky's
theorem in~\cite{Ha} and the proof in Douady-Hubbard's paper~\cite{DH} directly. Actually, the first two theorems
play crucial roles in Bonnot and Yampolsky's work in~\cite{BY} and the last one is crucial in~\cite{Se}.
But in the sub-hyperbolic semi-rational branched covering case, one works on an infinite-dimensional Teichm\"uller space of the Riemann sphere minus
a set of finitely many points and a set of finitely many topological disks.
Thus one can not use Minsky's theorem in~\cite{Mi}, Haissinsky's
theorem in~\cite{Ha} and the proof in Douady and Hubbard's paper~\cite{DH} directly.
{\em A major work in this paper is to overcome this difficulty}.

This paper is organized as follows. In Section 2, we review definitions
of post-critically finite rational maps and sub-hyperbolic semi-rational branched coverings.
We also review the combinatorial equivalence and the CLH-equivalence. We will state the existence
of the canonical Thurston obstruction for a sub-hyperbolic semi-rational branched covering
which is not CLH-equivalent to a geometrically finite rational map proved in~\cite{ChJ}.
In Section 3, we will decompose the two-sphere into several Riemann surfaces which are the Riemann sphere
minus some points and some disks. In the decomposition theorem (Theorem~\ref{decop}), we will prove that
under the action of $f$, each of these new Riemann surfaces is eventually periodic and at least one is periodic.
For each periodic Riemann surface, we will add a disk marked by a point to each hole and then extend it into a two-sphere. This will be worked out in Section 4. In the same section, we will prove our extension theorem (Theorem~\ref{ext}). This theorem says that for every periodic two-sphere with period $k>0$, we can define a new branched covering $\tilde{f}$ extending the original $f^{k}$ to the disks attached on all holes such that all marked points are super-attractive fixed point of $\tilde{f}$.
The new map $\tilde{f}$ is either a post-critically finite type branched covering or a sub-hyperbolic semi-rational type branched covering.
Finally, in Section 5, we will prove our main result (Theorem~\ref{mainth}). The main result says that if the new branched covering is a post-critically finite type branched covering with a hyperbolic orbifold, then it has no Thurston obstruction and is combinatorially equivalent to a rational map and that if the new branched covering is a sub-hyperbolic and semi-rational type branched covering with a hyperbolic orbifold, then it has no Thurston obstruction and is CLH-equivalent to a rational map.

\section{Post-critically finite and Sub-hyperbolic semi-rational branched coverings}

Let $\mathbb{S}^2$ be the topological two-sphere and $f:
\mathbb{S}^2\rightarrow \mathbb{S}^2$ be an orientation-preserving
branched covering of degree $d\geq 2$. Denote by $deg_xf$ the local
degree of $f$ at $x$. The set of branched points is
$$
\Omega_f=\{c\in \mathbb{S}^2: deg_cf\geq 2\}.
$$
The post-critical set $P_f$ by definition is
\[
P_f=\overline{\bigcup_{n>0}f^n(\Omega_f)}
\]
If $P_f$ is a finite set, we call $f$ a post-critically
finite branched covering. If $P_{f}$ is infinite but the accumulation set $P_{f}'$ is finite,
we call $f$ a geometrically finite. In this case, $P_{f}'$ consists of a finitely many of periodic orbits.

Suppose $f: \mathbb{S}^2\to \mathbb{S}^2$ is a branched covering. Define the signature
$\nu_f: \mathbb{S}^2\rightarrow \mathbb{N}\cup\{\infty\}$ as
\[
\nu_f(x)=
\begin{cases}
1,& x\notin P_f;\\
lcm\{ deg_y f^n\;|\;\; \forall n>0,\; f^n(y)=x\}, & otherwise.
\end{cases}
\]
The orbifold associated to $f$ is $\mathcal{O}_f=(\mathbb{S}^2,\nu_f)$
and the Euler characteristic of $\mathcal{O}_f$ is defined by
\[
\chi(\mathcal{O}_f)=2-\sum_{x\in\mathbb{S}^2}\left(1-\frac{1}{\nu_f(x)}\right).
\]
It is well known that $\chi(\mathcal{O}_f)\leq 0$ for any
post-critically finite branched covering $f$. The orbifold
$\mathcal{O}_f$ is called hyperbolic if $\chi(\mathcal{O}_f)<0$ and
parabolic if $\chi(\mathcal{O}_f)=0$.

\medskip
\begin{defi}[Combinatorial equivalence]~\label{ce}
Suppose $f$ and $g$ are post-critically finite branched coverings.
We say that they are combinatorially equivalent if there exist a pair of homeomorphisms $\phi$ and $\psi$ of
the two-sphere $\mathbb{S}^{2}$ such that
\begin{itemize}
\item[(1)] $\psi$ is isotopic to $\phi$ rel $P_f$ and
\item[(2)] $\phi\circ f=g\circ \psi$.
\end{itemize}
\end{defi}

Now let us turn to sub-hyperbolic semi-rational branched coverings defined in~\cite{CJ}.

\medskip
\begin{defi}~\label{shsr}
Let $f: \overline{\mathbb{C}}\rightarrow \overline{\mathbb{C}}$ be a
geometrically finite branched covering of degree $d\geq 2$. We say $f$ is a
sub-hyperbolic semi-rational branched covering if for any $a\in
P_f'$ of period $p\geq 1$, there is an open neighborhood $U$ of $a$,
such that $f$ is holomorphic in $U$, and
\begin{enumerate}
\item if $deg_af^p=1$, then
 \[
 f^p(z)=a+\lambda(z-a)+o(|z-a|)
 \]
for $z\in U$ and $0<|\lambda|<1$ is some constant and
\item if $deg_af^p>1$, then
 \[
 f^p(z)=a+\alpha(z-a)^k+o(|z-a|^k)
 \]
for some $z\in U$ and $\alpha\neq 0$ is some constant.
\end{enumerate}
\end{defi}

Corresponding to the combinatorial equivalence in the post-critical finite case,
the following CLH-equivalence (combinatorial and locally holomorphical equivalence)
was naturally introduced in~\cite{CJ}.

\medskip
\begin{defi}[CLH-equivalence]
Suppose $f$ and $g$ are two sub-hyperbolic semi-rational branched coverings.
We say that they are CLH-equivalent if there exists a pair of homeomorphisms $\phi$ and $\psi$ of $\overline{\mathbb{C}}$
such that
\begin{itemize}
\item[(1)] $\psi$ is isotopic to $\phi$ rel $P_f$,
\item[(2)] $\phi\circ f=g\circ \psi$, and
\item[(3)] $\phi|_{U}=\psi|_{U}$ is holomorphic on some open set $U\supset P_f'$.
\end{itemize}
\end{defi}

We have the following useful lemma proved in~\cite{ZJ} for a sub-hyperbolic semi-rational branched covering $f$.

\medskip
\begin{lem}[Shielding ring lemma]
There is a finite collection $\{D_i\}$ of open disks and a finite collection of open annuli $\{A_i\}$ such that
\begin{itemize}
\item[i)] $a_i\in D_i$,
\item[ii)] every $\partial D_i$ is a real analytic curve,
\item[iii)] $\overline{D_i}\cap\overline{D_j}=\emptyset$ for $i\neq j$,
\item[iv)] for each $i$, $A_i$ is an annulus attaching $D_i$ from outside such that $\overline{A_i}\cap P_f=\emptyset$,
\item[v)] $f$ is holomorphic on $\overline{D_i}\cup A_i$, and
\item[vi)] every $f(\overline{D_i}\cup A_i)$ is contained in some $D_j$.
\end{itemize}
\end{lem}

Thus we define
\[
D=\cup D_i \quad \hbox{and} \quad  P_1=P_f\backslash D
\]
and define
\begin{equation}~\label{set}
Q=P_1\cup \overline{D}\quad \hbox{and}\quad  X=\partial Q=P_1\cup \partial D.
\end{equation}

Now we suppose $f$ is either a post-critically finite or sub-hyperbolic semi-rational branched covering.
If $f$ is post-critically finite, we let $E=P_{f}$ and if $f$ is sub-hyperbolic semi-rational,
we let $E=Q$ as defined in (\ref{set}). In both of the cases, assume $0, 1, \infty \in E$,
the Teichm\"uller space $T_f$ for $f$ is the Teichm\"uller space modeled on $(\overline{\mathbb{C}}\backslash E, \partial E)$.
There is a natural induced map $\sigma_{f}: T_{f}\to T_{f}$ as follows. A point $\tau$ can be thought as the equivalent classes $[\mu]_{E}$ of
conformal structures $\mu$ on $\overline{\mathbb{C}}$ such that $\mu|E\equiv 0$. Two conformal structures $\mu$ and $\nu$ are said to be equivalent, denote as $\mu\sim_{E}\nu$,  if $(w^{\nu})^{-1}\circ w^{\mu}$ is isotopic to the identity rel $E$, where $w^{\mu}$ and $w^{\nu}$ are normalized (i.e., fixing $0, 1,\infty$)
quasiconformal homeomorphisms with Beltrami coefficients $\mu$ and $\nu$. For any $\tau=[\mu]\in T_{f}$, let $f^{*}\mu$ be the pull-back conformal structure
of $\mu$ by $f$. We take the proof of the following lemma from~\cite{J,ChJ1}.

\vspace*{5pt}
\begin{lem}~\label{pullback}
If $\mu\sim_{E} \nu$, then $f^{*}\mu\sim_{E} f^{*}\nu$.
\end{lem}

\begin{proof}
From the assumption, $\mu\sim_{E} \nu$, we can find a continuous map $H(t, z): [0,1]\times \overline{\mathbb{C}}\to \overline{\mathbb{C}}$ such that
\begin{itemize}
\item[1)] $H(0, z)=z$ for all $z\in \overline{\mathbb{C}}$;
\item[2)] $H(1, z)=(w^{\mu})^{-1} \circ w^{\nu}(z)$ for all $z\in \overline{\mathbb{C}}$; and
\item[3)] $H(t,z)=z$ for all $z\in E$ and all $0\leq t\leq 1$.
\end{itemize}

Since $f(\Omega_{f})\subseteq E$, we have that $\Omega_{f}\subseteq f^{-1}(E)$.
This implies that $f:\overline{\mathbb{C}}\setminus f^{-1}(E)\to \overline{\mathbb{C}}\setminus E$ is a covering map.
The homotopy $H(t,z): [0,1]\times (\overline{\mathbb{C}}\setminus E) \to\overline{\mathbb{C}} \setminus E$ rel $\partial E$ can be lift to a
homotopy $\widetilde{H}(t,z): [0,1]\times (\overline{\mathbb{C}}\setminus f^{-1}(E)) \to \overline{\mathbb{C}}\setminus f^{-1}(E)$ rel $\partial f^{-1}(E)$ such that
$$
H(t, f(z)) = f(\widetilde{H}(t, z)), \quad \forall z\in \overline{\mathbb{C}}\setminus f^{-1}(E), \;\; 0\leq t\leq 1
$$
and
$$
\widetilde{H} (t, z)=z, \quad \forall z\in \partial f^{-1}(E),\;\; 0\leq t\leq 1.
$$
Define $\widetilde{H}(t, z)=z$ for all $z\in f^{-1}(E)$ and $0\leq t\leq 1$. Then the new defined map,
which we still denote as $\widetilde{H}$, is a continuous map
$\widetilde{H} (t, z): [0,1]\times \overline{\mathbb{C}}\to \overline{\mathbb{C}}$ such that
\begin{itemize}
\item[a)] $\widetilde{H}(z, 0)=z$ for all $z\in \overline{\mathbb{C}}$ and
\item[b)] $\widetilde{H}(z,t)=z$ for all $z\in E$ and all $0\leq t\leq 1$.
\end{itemize}
Therefore, it is a homotopy from the identity to $\widetilde{H}_{1}(z)=\widetilde{H}(1, z)$.

Let $H_{1}(z)=H(1, z)$.
Since $H_{1}\circ f=f\circ \widetilde{H}_{1}$ and $H_{1}=(w^{\mu})^{-1} \circ w^{\nu}$ is quasiconformal, $\widetilde{H}_{1}$ is quasiconformal.

Now by using two commuting equations,
$$
(w^{\mu})^{-1} \circ w^{\nu}(z)\circ f=f\circ \widetilde{H}_{1}\quad \hbox{and}\quad g\circ w^{f^{*}\mu} = w^{\mu}\circ f,
$$
where $g$ is holomorphic,
we have that
$$
g\circ w^{f^{*}\mu}\circ \widetilde{H}_{1}=w^{\mu}\circ f\circ \widetilde{H}_{1} = w^{\mu}\circ (w^{\mu})^{-1} \circ w^{\nu}(z)\circ f=w^{\nu}(z)\circ f.
$$
Since $g$ is holomorphic,
$$
\mu_{w^{f^{*}\mu}\circ \widetilde{H}_{1}} =\mu_{g\circ w^{f^{*}\mu}\circ \widetilde{H}_{1}} =\mu_{w^{\nu}(z)\circ f}=f^{*}\nu.
$$
Since both $w^{f^{*}\mu}\circ \widetilde{H}_{1}$ and $w^{f^{*}\nu}$ fix $0, 1, \infty$, we get
$$
w^{f^{*}\mu}\circ \widetilde{H}_{1}= w^{f^{*}\nu}.
$$
In other words,
$$
\widetilde{H}_{1}=(w^{f^{*}\mu})^{-1}\circ w^{f^{*}\nu}.
$$
Thus $\widetilde{H} (t, z): [0,1]\times \overline{\mathbb{C}}\to \overline{\mathbb{C}}$ is a homotopy from the identity to $(w^{f^{*}\mu})^{-1}\circ w^{f^{*}\nu}$ rel
$f^{-1}(E)$. But $E\subseteq f^{-1}(E)$, the last statement implies that $f^{*}\mu\sim_{E} f^{*}\nu$. This completes the proof.
\end{proof}

Thus we can define a map $\sigma_{f}(\tau)=[f^{*}\mu]_{E}$ for any $\tau=[\mu]_{E}$ from $T_{f}$ into itself. It is a holomorphic map.
Each element $\tau=[\mu]_{E}\in T_{f}$ determines a complex structure on $(\overline{\mathbb{C}}\backslash E, \partial E)$.
For a non-peripheral curve in $(\overline{\mathbb{C}}\backslash E, \partial E)$,
let $l_\tau(\gamma)$ denote the hyperbolic length of the unique geodesic homotopic to $\gamma$
on the marked Riemann surface $(\overline{\mathbb{C}}\backslash E, \partial E)$ with the complex structure $\tau=[\mu]_{E}\in T_f$.
Given an element $\tau_{0}=[\mu_{0}]_{E}\in T_{f}$, we have a sequence $\tau_{n}=\sigma_{f}^{n} (\tau_{0})=[\mu_{n}]_{E}$, $n=0, 1, 2, \cdots$.

\medskip
\begin{defi}~\label{canTh1}
The set
$$
\Gamma_{c} =\{ \gamma \;|\; l_{\tau_{n}}(\gamma) \to 0, \;\; n\to \infty\}
$$
is called the canonical Thurston obstruction for $f$.
\end{defi}

Pilgrim proved that
\medskip
\begin{thm}[\cite{Pi1}]~\label{pthm}
A post-critically finite branched covering $f$ with hyperbolic orbifold is not combinatorially equivalent to a rational map if and only if
$\Gamma_{c}\neq\emptyset$.
\end{thm}

Furthermore, in~\cite{ChJ}, we have proved that

\medskip
\begin{thm}[\cite{ChJ}]~\label{cjthm}
A sub-hyperbolic semi-rational branched covering $f$ is not CLH-equivalent to a rational map if and only if
$\Gamma_c\neq\emptyset$.
\end{thm}

A simple closed curve $\gamma\subset \mathbb{S}^2\backslash E$ is
called non-peripheral if each component of $\mathbb{S}^2\backslash
\gamma$ contains at least two points of $E$. A multicurve $\Gamma$
on $\mathbb{S}^2\backslash E$ is a set of disjoint, non-homotopic,
non-peripheral simple closed curves on $\mathbb{S}^2\backslash E$.
A multicurve $\Gamma$ is called $f$-stable (or stable briefly) if for any $\gamma\in\Gamma$, all the non-peripheral
components of $f^{-1}(\gamma)$ are homotopic to elements of
$\Gamma$. For a given $f$, we can define Thurston linear
transformation $f_\Gamma: \mathbb{R}\rightarrow \mathbb{R}$ as
follow: Let $\gamma_{i,j,\alpha}$ be the components of
$f^{-1}(\gamma_j)$ homotopic to $\gamma_i$ in
$\mathbb{S}^2\backslash P_f$. Define
\[
f_\Gamma(\gamma_j)=\sum_{i,\alpha}\frac{1}{d_{i,j,\alpha}}\gamma_i,
\]
where $d_{i,j,\alpha}=deg(f: \gamma_{i,j,\alpha}\rightarrow
\gamma_j)$. We can write $f_\Gamma(\Gamma)=A_\Gamma\Gamma$ where
$A_\Gamma$ (or denoted by $A_{\Gamma,f}$) is a $n\times n$
non-negative matrix if $\Gamma=\{\gamma_1,\cdots,\gamma_n\}$. Let
$\lambda(f,\Gamma)\geq 0$ (or $\lambda(\Gamma)\geq 0$ briefly) be the spectral
radius of $A_\Gamma$. For a
stable multicurve $\Gamma$, it is easy to see
$(f^n)_\Gamma=(f_\Gamma)^n$.

\medskip
\begin{defi}~\label{THOb}
A stable multicurve $\Gamma$ is called a Thurston obstruction if $\lambda (\Gamma)\geq 1$.
\end{defi}

It is known that if $\Gamma_{c}\not=\emptyset$, then it is a Thurston obstruction as proved in~\cite{Pi1} for the post-critically finite case
and in~\cite{ChJ} for the sub-hyperbolic semi-rational case. Furthermore, we have also the following from these two papers.

\medskip
\begin{lem}~\label{bigconst}
If $\Gamma_{c}\neq \emptyset$, for any $\tau_{0}\in T_{f}$, there is a constant $L=L(\tau_0)>0$ such that
for any non-peripheral simple closed curve $\gamma\notin\Gamma_c$,
$l_{\tau_n}(\gamma)\geq L>0$ for all $n>0$.
\end{lem}

Moreover, if $\Gamma_c\neq \emptyset$, then it is also a simple obstruction in the meaning that no permutation of the
curves in $\Gamma_c$ such that the matrix $A_{\Gamma_{c}}$ be written in the block form
$$
\left(
\begin{matrix}
M_{11} & 0\cr
M_{21} & M_{22}
\end{matrix}
\right)
$$
Thus $\Gamma_{c}$ is stable and full, that is, $f^{-1} (\Gamma_{c}) =\Gamma_{c}$.

In the following sections, we will focus on the class of all sub-hyperbolic semi-rational branched coverings such that
$\Gamma_c\neq \emptyset$, in other words, such that they are not CLH-equivalent to rational maps.
However, we will prove that after ``decomposition" and ``extension",
parts of these maps can be still combinatorially equivalent or CLH-equivalent to rational maps under one condition that the extended map has hyperbolic orbifold.
We will divide our idea into three sections, decomposition, extension, and equivalence.

\section{Decomposition}
Suppose $f$ is a sub-hyperbolic semi-rational branched covering and
$$
\Gamma_c=\{\gamma_1,\cdots,\gamma_n\}\not=\emptyset
$$
is its canonical Thurston obstruction.
Then $f$ is not CLH-equivalent to a rational map. Let $Q$ be the set defined in (\ref{set}).
Suppose $A_{0,i}\ (i=1,\cdots,n)$ are a collection of disjoint annuli
whose core curves are $\gamma_i\ (i=1,\cdots,n)$ respectively.
Set
\[
{\mathcal A}_0=\bigcup_{i=1}^{n} A_{0,i}.
\]
Let
$$
{\mathcal A}_1 =\bigcup_{i=1}^{m} A_{1, i}
$$
be the union of preimage of elements of ${\mathcal A}_0$
such that every component of ${\mathcal A}_1$
is homotopic to some component in ${\mathcal A}_0$ rel $Q$.
By the same method as Pilgrim used in~\cite{Pi2}, we get

\medskip
\begin{pro}~\label{homt}
There exists a homeomorphism $h: (\overline{\mathbb{C}},Q)\rightarrow (\overline{\mathbb{C}},Q)$ isotopic to id rel $Q$
such that $\widetilde{f}=f\circ h$ satisfied the followings:
\begin{itemize}
\item[(1)] every curve $\gamma_i\in \Gamma_{c}$ is a core curve of some annulus $A_{0,i}\in {\mathcal A}_0$;
\item[(2)] every $A_{1,k}\in {\mathcal A}_{1}$ is a component of the preimage of some $A_{0,j}\in {\mathcal A}_{0}$
and homotopic to some $A_{0,i}\in {\mathcal A}_0$, denote by $A_{1, ji, \alpha}$;
\item[(3)] for each given $A_{i}\in {\mathcal A}_{0}$, the union $A_{1, i,i}=\cup_{j,\alpha} A_{1, ji, \alpha}$
of all components of ${\mathcal A}_1$ homotopic to $\gamma_i$ is contained inside $A_{0,i}$;
\item[(4)] two outmost annuli from $A_{1,i, i}$ share their outer boundary curves with $A_{0,i}$; and
\item[(5)] restricted to a boundary curve $\chi$ of $A_{0,i}$, the map $f: \chi\rightarrow f(\chi)$, which is a boundary curve of $A_{0,j}$,
is given by $z\mapsto z^d: \mathbb{S}^1\rightarrow \mathbb{S}^1$ for some $d$ up to a homeomorphism conjugation.
\end{itemize}
\end{pro}

The set of annuli ${\mathcal A}_{0}$ satisfy (1)-(5) in Proposition~\ref{homt} is called a standard form.
Without loss of generality, we always assume $f$ itself has a standard form with respect to its canonical Thurston obstruction
$\Gamma_c$.

\medskip
\begin{defi}
We call each component $A_{0,i}$ of ${\mathcal A}_0$ a thin part and each component of $\overline{\mathbb{C}}\backslash {\mathcal A}_0$
a thick part.
\end{defi}

Let
$$
\mathcal{B}_{0}=\{P_{i}^{0}\}_{i=1}^{n_{0}}
$$
be the collection of all thick parts.
Pull-back $\mathcal{B}_{0}$ by $f^{k}$, we have
\[
\mathcal{B}_k=f^{-k} (\mathcal{B}_{0}) =\{ P_i^k\}_{i=1}^{n_{k}}
\]
with
\[
\bigcup_{i=1}^{n_0}\{f^{-k}(P_i^0)\}= \bigcup_{i=1}^{n_{k}} P_{i}^{k}.
\]
Then each element of $\mathcal{B}_k$ belongs to one and only one of the following four classes:
\begin{enumerate}
\item Disk component $\mathcal{D}$ if it is a topological disk and $\mathcal{D}\cap P_f=\emptyset$.
\item Punctured disk component $\mathcal{P}$ if it is a topological disk and $\sharp(\mathcal{P}\cap P_f)=1$.
\item Annulus component $\mathcal{A}$ if it is an annulus and $\mathcal{A}\cap P_f=\emptyset$.
\item Complex component $\mathcal{C}$ if it is not in (1), (2), and (3).
\end{enumerate}
Since all elements of $\Gamma_c$ are non-peripheral and non-homotopic each other,
all thick parts $P_0^0,\cdots,P_{n_0}^{0}$ are complex components.

For each thick part $P_i^0\in \mathcal{B}_0$ and each $k\geq 1$,
remember that ${\mathcal A}_0$ is in the standard form,
there exists an unique component of $\mathcal{B}_k$, denote by $P_i^k$, such that
\begin{itemize}
\item[i)] each component of $\partial P_i^k$ is either peripheral or some
component of $\partial P_i^0$ and
\item[ii)] each component of $\partial P_i^0$ is some component of $\partial P_i^k$.
\end{itemize}
Therefore, for each $P_{i}^{k}$, $f(P_i^k)$ is also a complex component and
if $f(P_i^k)=P_j^{k-1}$, then $f(P_i^l)=P_j^{l-1}$ for any $l\geq 1$.
And furthermore, consider the set of integers $I=\{1, \cdots, n_{0}\}$.
For each $i\in I$, let $k\geq 1$ be any integer, we have a unique $j\in I$ such that $f(P_{i}^{k})\subset P_{j}^{k-1}$.
This defines a self map $\tau: I\to I; \tau (i)=j$. Since $\Gamma_{c}$ is stable and full, we have that

\medskip
\begin{pro}
Each number $i\in I$ is a preperiodic point under iteration of $\tau$ and at least one number in $I$ is a periodic point of $\tau$.
\end{pro}

\vspace*{5pt}
\begin{defi}~\label{preperiodic}
If $i\in I$ is preperiodic, then we say that the corresponding thick part $P_{i}^{0}$ is preperiodic
and if $i\in I$ is periodic, we say that the corresponding thick part $P_{i}^{0}$ is periodic.
\end{defi}

Thus we get the following decomposition theorem.

\medskip
\begin{thm}[Decomposition]~\label{decop}
Suppose $f$ is a sub-hyperbolic semi-rational branched covering not
CLH-equivalent to a rational map. Then ${\mathbb S}^2$ can be
decomposed into thin parts and thick parts according to the canonical
Thurston obstruction $\Gamma_{c}$. Furthermore, each thick part is
eventually periodic and at least one thick part is periodic.
\end{thm}

\section{Extension}
Suppose $f: \mathbb{S}^2\rightarrow \mathbb{S}^2$ is a branched
covering and $P$ is a point set in $\mathbb{S}^2$. We define
a sub-hyperbolic semi-rational {\em type} branched covering (with respect to $P$).

\begin{defi}
We call a branched covering $f$ sub-hyperbolic semi-rational
type if
\begin{itemize}
\item[(1)] $f$ satisfies Definition~\ref{shsr} by replacing
$P_f$ to $P$;
\item[(2)] $\sharp P=\infty$;
\item[(3)] $P_f\subseteq P$;
\item[(4)] $f(P)\subseteq P$.
\end{itemize}
\end{defi}

We also define a post-critically finite {\em type} branched covering (with respect to $P$).

\begin{defi}
We call a branched covering $f$ post-critically finite type
type if
\begin{itemize}
\item[(1)] $\sharp P<\infty$;
\item[(2)] $P_f\subseteq P$;
\item[(3)] $f(P)\subseteq P$.
\end{itemize}
\end{defi}

We say that a sub-hyperbolic semi-rational type or post-critically
finite type branched covering has hyperbolic orbifold if
$\chi(\mathcal{O}_f)<0$.
We would like to note that a sub-hyperbolic semi-rational {\em type} branched covering
may not have a hyperbolic orbifold even $\sharp P=+\infty$

A sub-hyperbolic semi-rational type branched covering can be constructed
from a sub-hyperbolic semi-rational branched covering or from a
post-critically finite branched covering by adding finite or
infinite number points to $P_f$. Similarly, a post-critically finite
type branched covering can be constructed from a post-critically
finite branched covering by adding finite number points to $P_f$.

Suppose $P_{0}^{0}$ is a periodic thick part. Suppose $k>0$ is the period.
Suppose $\gamma_1,\cdots,\gamma_p$ are boundary curves of $P_0^0$ and
$\gamma_1,\cdots,\gamma_p,\beta_1,\cdots,\beta_q$ are boundary
curves of $P_0^k$ where $\beta_j$ ($j=1,\cdots,q$) are peripheral
curves. For any $\beta_j$, it must be a component of
$f^{-k}(\gamma_i)$ for some $\gamma_i$. Denote
\[
P_0^0\backslash P_0^k=\bigcup_{j=1}^q D(\beta_j),\ \
\overline{\mathbb{C}}\backslash P_0^0=\bigcup _{i=1}^p D(\gamma_i).
\]
Let
\[
d_{\beta_j}=deg(f^k: \beta_j (\subset f^{-k}(\gamma_i))\rightarrow \gamma_i),\
\ d_{\gamma_i}=deg(f^k: \gamma_i\rightarrow f^k(\gamma_i)).
\]
Define a new branched covering map by
\begin{equation}~\label{newmap}
\widetilde{f}=\widetilde{f}_{P_0^{0}}=
\begin{cases}
f^k (z),\ \ \ z\in P_0^k\\
\varphi_j\circ z^{d_{\beta_j}}\circ \psi_j,\ \ \ z\in D(\beta_j)\ \ (j=1,\cdots,q)\\
\varphi_i\circ z^{d_{\gamma_i}}\circ \psi_i,\ \ \ z\in D(\gamma_i)\
\ (i=1,\cdots,p)
\end{cases}
\end{equation}
where $\psi_j, \varphi_j^{-1}$ are homeomorphisms from $D(\beta_j)$ and $D(\gamma_i)$ to the unit disk
$\mathbb{D}$, respectively, and $\psi_i, \varphi_i^{-1}$ are homeomorphisms from $D(\gamma_i)$ and $D(f^{k}(\gamma_i))$ to the unit disk
$\mathbb{D}$, respectively, such that $\widetilde{f}$ is continuous.
For each disk $D(\gamma_i)$, we mark a point $z_{i}$.
If $D(\beta_j)$ contains a point, say $z^*$, belonging to $P_f$ and
$f^k(\beta_j)=\gamma_i$, we can select $\varphi_j, \varphi_{i}$ and $\psi_j, \psi_{i}$ such
that $\widetilde{f}(z^*)=z_i$.
Also, if $f^k(\gamma_i)=\gamma_k$, we can select $\varphi_i,\varphi_k,\psi_i,\psi_k$
such that $\widetilde{f}(z_i)=z_k$. Thus we have that
$$
P_{\widetilde{f}}\subseteq P=(P_f\cap P_0^0)\cup (\cup_{j=1}^p \{z_j\}).
$$
and
$$
\widetilde{f}(P_{\widetilde{f}})\subseteq
P_{\widetilde{f}}\quad \hbox{and}\quad \widetilde{f}(P)\subseteq
P.
$$

The above process extends $f^{k}$ for every periodic thick part of period $k\geq 1$ to a new branched covering
$\widetilde{f}: \overline{\mathbb{C}}\rightarrow \overline{\mathbb{C}}$. Then $\widetilde{f}$ is either
a post-critically finite branched covering or a sub-hyperbolic semi-rational branched covering.
Without causing any confusion, we use now $TP$ to denote a periodic thick part $P_{i}^{0}$ of period $k\geq 1$ and $TP'$ to denote
$P_{i}^{k}\subset P_{i}^{0}$. We state the above process into a theorem.

\medskip
\begin{thm}[Extension]~\label{ext}
Suppose $f$ is a sub-hyperbolic semi-rational branched covering such that $\Gamma_{c}\not= \emptyset$.
Suppose $TP$ is a periodic thick part of period $k\geq 1$. Then $TP'$ and $TP$ can be extended to the $2$-sphere
by adding finitely many disks with marked points and the map $f^{k}: TP'\to TP$ can be extended to a branched covering $\tilde{f}$
of the $2$-sphere which is either a post-critically finite type branched covering or a sub-hyperbolic semi-rational
type branched covering.
\end{thm}

Furthermore, if we start with a quasiregular branched covering $f$, then we can also extend to a quasiregular $\tilde{f}$.

\section{Realization}
Based on our decomposition theorem (Theorem~\ref{decop})
and our extension theorem (Theorem~\ref{ext}), we now state our
main theorem.

\medskip
\begin{thm}[Main Theorem]~\label{mainth}
Every $\widetilde{f}$ is either a post-critically finite type branched covering
or a sub-hyperbolic semi-rational type branched covering.
If the orbifold associate to $\widetilde{f}$ is hyperbolic,
in the post-critically finite type case,  $\widetilde{f}$
is combinatorially equivalent to a rational map; in the sub-hyperbolic semi-rational type case, $\widetilde{f}$ is
CLH-equivalent to a rational map. Moreover, in the both cases, the realized rational map
is unique up to conjugation of an automorphism of the Riemann sphere.
\end{thm}

To prove the main theorem (Theorem~\ref{mainth}),
we first prove three propositions (Propositions~\ref{samebig},~\ref{cp}, and~\ref{relationship}); each of them has its own interest.

We first quote a theorem from~\cite[Theorem 6.3]{DH}.

\vspace*{5pt}
\begin{thm}~\label{collar}
Let $R$ be a Riemann surface with conformal structure $\tau$. Let $\alpha_{1}, \cdots, \alpha_{n}$ be
disjoint simple closed geodesics of length $l_{1}, \cdots, l_{n}$. Then there exist in $R$ disjoint annuli
$A(\alpha_{i})$ ($i=1, \cdots, n$) which are collars of $\alpha_{i}$ ($i=1, \cdots, n$) such that each modulus
$\mod(A(\alpha_{i}))$ satisfies
$$
\frac{\pi}{2l_{i}}-1\leq \mod(A(\alpha_{i})) \leq \frac{\pi}{2l_{i}} \quad (i=1, \cdots, n).
$$
\end{thm}

\medskip
\begin{pro}~\label{samebig}
Suppose $R_1$ is a Riemann surface, which is the Riemann sphere minus a set $E$ consisting of finite number of points and finite number of disks,
with complex structure $\tau_1=[\mu_{1}]$. Suppose $\gamma_1,\cdots,\gamma_n$ are non-peripheral non-homotopic
simple closed curves on $R_1$ with $l_{\tau_1}(\gamma_i)<\varepsilon$ ($\varepsilon$ sufficiently
small). Let $R_2$ be a Riemann surface with complex structure
$\tau_2=[\mu_{2}]$ obtained from $R_1$ by cutting along $\gamma_i$ and
capping every hole by a puncture disk. If there exists a constant
$K>0$ such that for every non-peripheral simple closed curve
$\beta$ other than $\gamma_1,\cdots,\gamma_n$,
$l_{\tau_1}(\beta)\geq K$, then there exists a constant
$\tilde{K}=\tilde{K} (K,\varepsilon)>0$ such that for every non-peripheral simple
closed curve $\widetilde{\beta}$ of $R_{2}$, $l_{\tau_2}(\widetilde{\beta})\geq \tilde{K}$.
\end{pro}

\medskip
\begin{rem}
In the proposition, the term ''capping a puncture disk'' means that first extend $\mu_{1}$ to the Riemann sphere by defining $\mu_{1}=0$ on $E$.
Every curve $\gamma_{i}$ cuts the Riemann sphere into two disks, one of them is disjoint with other $\gamma_{j}$. Mark a point in this disk and set $\mu_{2}=0$ on this punctured disk. Then define $\mu_{2}=\mu_{1}$ on the rest of $R_{2}$.
The Riemann surfaces $R_1$ and $R_2$ may be different type Riemann surfaces.
\end{rem}

\begin{proof}
First, we use $R^*$ to denote the part of $R_{1}$ and $R_{2}$ by cutting
along $\gamma_i\ (i=1,\cdots,n)$ before we capped all holes by puncture disks.
Then $R_1$ has the same complex structure with $R_2$ on $R^*$.

Let $\alpha_{i}$ be the closed geodesic in $R_{1}$ homotopic to $\gamma_i$.
Since $l_{\tau_1}(\gamma_i)<\varepsilon$, the hyperbolic length $l_{\tau_1}(\alpha_{i})<\varepsilon$.
Suppose $A(\alpha_i)$ ($i=1,\cdots,n$) is the collar of $\alpha_i$ in $R_{1}$ in Theorem~\ref{collar}.
Then $A(\alpha_i)$ ($i=1,\cdots,n$) are pairwise disjoint and homotopic to $\gamma_i$.

Denote the boundary curves of $A(\alpha_i)$ by $\alpha_i^+$ and $\alpha_i^-$ respectively.
Since $\varepsilon$ small, either $\alpha_i\cap\gamma_i=\emptyset$ or $\alpha_i=\gamma_i$.
Without loss of generality, we always assume $\alpha_i^-$ lies in the same side of
$\gamma_{i}$ with respect to $\alpha_i$.
The following three cases can happened:
\begin{itemize}
\item[1.] $\alpha_i=\gamma_i$.
\item[2.] $\alpha_i\cap R^*=\emptyset$.
\item[3.] $\alpha_i\subset R^*$.
\end{itemize}
Then we have that
\begin{itemize}
\item[1)] In the first and the second cases, let $A_{\gamma_i,\alpha_i^{+}}$ be the annulus
with boundary curves $\gamma_i,\alpha_i^{+}$. Let $\gamma_i^{-}$ be the core curve of this annulus.
Define $A_{\gamma_i^-,\gamma_i}$ as the annulus with boundaries $\gamma_i^-,\gamma_i$. This annulus is outside $R^*$ and attaching $R^*$.
%
%
%
\item[2)] In the third case, let $A_{\gamma_i,\alpha_i^{+}}$ be the annulus with boundary curves
$\alpha_i^+,\gamma_{i}$. Let $\gamma_i^-$ be the core curve of this annulus.
Define $A_{\gamma_i^{-},\gamma_i}$ as the annulus with boundary curves $\gamma_i^{-},\gamma_i$.
This annulus is inside $R^{*}$ and attaching $\gamma_i$.
%
%
\end{itemize}
Since the hyperbolic length $l_{\tau_1}(\alpha_{i})<\varepsilon$, we have a constant $K_{1}=K_{1} (\epsilon)>0$ such that
the modulus of $A_{\gamma_i^-,\gamma_i}$ is greater than $K_{1}$.

We assume that we are in the second case and give a detailed proof.
In the first and third cases, the proof is similar.

Let $\widetilde{\gamma}_i^-, \widetilde{\gamma}_i$ be two non-homotopic zero closed curves in the punctured disk of $R_2$ we capped
such that they are homotopic each other in $R_{2}$. Let $A_{\widetilde{\gamma}_i^-,\widetilde{\gamma}_i}$ be the annulus with boundary curves
$\widetilde{\gamma}_i^-, \widetilde{\gamma}_i$. This annulus is a non-homotopic zero annulus in the punctured disk we capped, therefore, it is a
non-homotopic zero annulus in $R_{2}$. Suppose the modulus of $A_{\widetilde{\gamma}_i^-,\widetilde{\gamma}_i}$
equals to the modulus of $A_{\gamma_i^-,\gamma_i}$. Thus the modulus of $A_{\widetilde{\gamma}_i^-,\widetilde{\gamma}_i}$ is greater than $K_{1}$.

Suppose $\widetilde{\beta}$ is an non-peripheral simple closed
geodesic in $R_2$. Suppose the hyperbolic length $l_{\tau_2}(\widetilde{\beta})=L$. Let
$A(\widetilde{\beta})$ be the collar of $\widetilde{\beta}$ in
$R_2$. We have
$$
\frac{\pi}{2}\frac{1}{l_{\tau_2}(\widetilde{\beta})}-1\leq  \mod(A_{\widetilde{\beta}}) \leq \frac{\pi}{2}\frac{1}{l_{\tau_2}(\widetilde{\beta})}.
$$
Suppose the boundary curves of $A(\widetilde{\beta})$ are $\widetilde{\beta}^+$ and
$\widetilde{\beta}^-$, respectively.
The annulus $A(\widetilde{\beta})$ is cut into two annuli by $\widetilde{\beta}$, which we denote as
$A_{\widetilde{\beta}^+,\widetilde{\beta}}, A_{\widetilde{\beta},\widetilde{\beta}^-}$.
Let $\widetilde{\beta}^{0,+}, \widetilde{\beta}^{0,-}$ be the core curves of these two sub-annuli.
When these two curves are viewed in $R_{1}$, we denote
as $\beta^{0,+}, \beta^{0,-}$. Similarly, we use $\beta$ to denote the corresponding curve
$\widetilde{\beta}$ when it is viewed as a curve on $R_1$.

First, we consider the case $\widetilde{\beta}^{0,+}\cap
\widetilde{\gamma}_i^-=\emptyset$ for all $i=1,\cdots,n$. Then $\beta^{0,+}\cap
\gamma_i^-=\emptyset$. Let $\zeta$ be the simple closed geodesic in $R_{1}$ homotopic to $\beta$.
Then the collar $A(\zeta)$ of $\zeta$ containing $A_{\beta^{0,+},\beta^{0,-}}$
and we have that
$$
\mod(A(\zeta))\geq \mod(A_{\beta^{0,+},\beta^{0,-}})=\mod(A_{\widetilde{\beta}^{0,+},\widetilde{\beta}^{0,-}})
=\frac{1}{2} \mod(A(\widetilde{\beta})).
$$
Now we have
$$
K\leq l_{\tau_1}(\beta)\leq \frac{\pi}{2} \frac{1}{\mod(A(\zeta))}\leq \frac{\pi}{\mod(A(\widetilde{\beta}))}
\leq \frac{\pi}{\frac{\pi}{2}\frac{1}{l_{\tau_2}(\widetilde{\beta})}-1}=\frac{2\pi L}{\pi-2L}.
$$
This implies that
$$
L\geq \frac{\pi K}{2\pi+2K}.
$$

Now we are considering the second case that is $\widetilde{\beta}^{0,+}\cap
\widetilde{\gamma}_i^-\not=\emptyset$ for some $i$. Without loss generality, we assume $z_i=0$, $z_j=1$, $z_k=\infty$ for some $j\neq i$ and
$k\neq i$. Furthermore, we can assume the capping punctured disk $D_i$ lies in the disk $|z|<\frac{1}{2}$ and
$\widetilde{\gamma}_i=\{z:|z|=\frac{1}{2}\}$.
By using the Poincar\'e density $\lambda_{0,1}(z)$ of the thrice punctured complex plane
$\mathbb{C}_{0,1}=\overline{\mathbb{C}}\setminus \{0,1, \infty\}$, we can give a lower bound
on the hyperbolic length of $\widetilde{\beta}^{0,+}$ by considering $R_{2}$ as a sub-Riemann surface of the thrice punctured complex plane,
that is
\[
l(\widetilde{\beta}^{0,+})\geq \int^{\frac{1}{2}}_r \lambda_{0,1}(x)dx
\]
where $x$ is real and $0<r<1/2$. Now we use the property of $\lambda_{0,1}(z)$
to estimate $l(\widetilde{\beta}^{0,+})$ (see~\cite{Ah2} for this property).
Since
\[
\lambda_{0,1} (z)\geq \frac{|\zeta'(z)|}{|\zeta(z)|}\frac{1}{4-\log |\zeta(z)|}
\]
where $\zeta(z)=\frac{\sqrt{1-z}-1}{\sqrt{1-z}+1}$ and $\mbox{Re} \sqrt{1-z}>0$,
\[
l(\widetilde{\beta}^{0,+})\geq \int^{\frac{1}{2}}_r \frac{|\zeta'(z)|}{|\zeta(z)|}\frac{1}{4-\log |\zeta(z)|}|dz|
\]
\[
=\log\frac{4-\log(-(\zeta(r)))}{4-\log(-(\zeta(\frac{1}{2})))}.
\]
Since $\mod(A_{\widetilde{\gamma}_i^-,\widetilde{\gamma}_i})>K_1$,
there exists a $0<r_0 <\frac{1}{2}$ such that $r<r_0$ (in fact, we can take $r_0=\frac{1}{2e^{2K_1}}$).

Since
\[
(-\zeta(r))'=\left(\frac{1-\sqrt{1-r}}{1+\sqrt{1-r}}\right)'=\frac{2}{\sqrt{1-r}(1+\sqrt{1-r})^2}>0,
\]
the function $\log\frac{4-\log(-(\zeta(r)))}{4-\log(-(\zeta(\frac{1}{2})))}$ is decreasing with respect to $r$.
So
\[
l(\widetilde{\beta}^{0,+})\geq C=\log\frac{4-\log(-(\zeta(r_{0})))}{4-\log(-(\zeta(\frac{1}{2})))}.
\]
Since $\widetilde{\beta}^{0,+}$ is the core curve of $A_{\widetilde{\beta}^{+},\widetilde{\beta}}$, we have
\[
\mod(A_{\widetilde{\beta}^{+},\widetilde{\beta}})\geq \frac{1}{2} \mod(A_{\widetilde{\beta}})
\geq \frac{1}{2} \Big(\frac{\pi}{2}\frac{1}{L}-1\Big) =\frac{\pi-2L}{4L}.
\]
Then
\[
C\leq l(\widetilde{\beta}^{0,+})\leq \frac{\pi}{2}\frac{1}{\mod(A_{\widetilde{\beta}^{+},\widetilde{\beta}})}\leq \frac{4 L}{\pi-2L}.
\]
Combining the above result, we have
\[
L\geq \frac{\pi C}{4+2C}.
\]

Let
\[
\tilde{K}=\min\Big\{\frac{\pi K}{2\pi+2K}, \frac{\pi C}{4+2C}\Big\}.
\]
Then for any non-peripheral simple closed curve $\widetilde{\beta}$ on $R_{2}$, we have that
$$
l_{\tau_2}(\widetilde{\beta})\geq \tilde{K}.
$$
This completes the proof.
\end{proof}

\medskip
The following is a key step in our proof. Suppose
$$
\{p_n^m\} \quad (n=1,\cdots,\infty;\quad m=1,\cdots, M)
$$
is a sequence of points in $\overline{\mathbb{C}}$ and
$$
p_n^m\rightarrow p_0^m, \quad n\rightarrow \infty.
$$
Assume there exists a finite collection $\{D_m\}_{m=1}^{M}$ of open disks (called holomorphic disks)
and a finite collection of open annuli $\{A_m\}_{m=1}^{M}$ satisfied
\begin{itemize}
\item[(a)] $p_0^m\in D_m$,
\item[(b)] $\overline{D}_m\cap \overline{D}_{m'}=\emptyset,\ (1\leq m\neq m'\leq M)$,
\item[(c)] for each $m$, $A_m$ is an annulus attaching $D_m$ from the outside such that for any
$n\in\mathbb{N},\ m=1,\cdots,M$, $\overline{A}_m\cap \{ \{ p_n^m\}_{n=1}^{\infty}\}_{m=1}^{M} =\emptyset$.
\end{itemize}
Recall
\[
D=\bigcup_{m=1}^M D_m,\ \ P_1=\bigcup_{m=1}^M\bigcup_{n=1}^\infty \{ \{ p_n^m\}_{n=1}^{\infty} \}_{m=1}^{M} \backslash D
\]
\[
Q=P_1\bigcup \overline{D},\ \ X=\partial Q=P_1\bigcup \partial D.
\]
For each holomorphic disk $D_m$, fix a point $p^m_*$ on the boundary $\partial D_m$. Set
\begin{equation}~\label{newset}
E=P_1\bigcup( \bigcup_{m=1}^M\{p_0^m,p_*^m\}).
\end{equation}
We {\sl assume} that $0, 1, \infty\in E$.

For any Beltrami coefficient $\mu$ on $\overline{\mathbb{C}}$ such that $\mu|Q=0$,
the surface $\overline{\mathbb{C}}\backslash Q$ with the complex structure $\mu$ is a hyperbolic Riemann surface, which we denote as
$RQ_{\mu}$. The surface $\overline{\mathbb{C}}\backslash E$ with the complex structure $\mu$ is also a hyperbolic Riemann surface, which we denote as
$RE_{\mu}$. Then $RQ_{\mu}$ is a sub-Riemann surface of $RE_{\mu}$.
For any non-peripheral simple closed curve $\beta$ in $\overline{\mathbb{C}}\backslash Q$, let $l_{\mu,Q}(\beta)$ denote the hyperbolic length of the unique simple closed geodesic in $RQ_{\mu}$ homotopic to $\beta$. For any non-peripheral simple closed curve $\beta$ in $\overline{\mathbb{C}}\backslash E$, let $l_{\mu,E}(\beta)$ denote the hyperbolic length of the unique simple closed geodesic in $RE_{\mu}$ homotopic to $\beta$.
To prove the following proposition, we need two lemmas. The first one can be found in any standard book in quasiconformal mapping theory (for example,~\cite{Ah1}) and the second one can be founded in~\cite{ZJ}. We give a proof of the second lemma in the sake of completeness of this paper.

\medskip
\begin{lem}~\label{mod}
Let $z\in \overline{\mathbb{C}}\backslash\{0,1,\infty\}$ be a point.
Let $H\subset \overline{\mathbb{C}}$ be an annulus which separates $\{0,1\}$ and $\{z,\infty\}$. Then
\[
\mod(H)\leq \frac{1}{2\pi}\log(|z|+1).
\]
\end{lem}

\medskip
\begin{lem}[\cite{ZJ}]~\label{zj}
There exists an $\eta>0$ such that for any Beltrami coefficient $\mu$ on $\overline{\mathbb{C}}$
with $\mu=0$ on $Q\cup \cup_{m=1}^{M} A_m$ and any non-peripheral simple closed curve
$\gamma\subset \overline{\mathbb{C}}\backslash E$ with $l_{\mu,E}(\gamma)<\eta$,
we have $\gamma\subset \overline{\mathbb{C}}\backslash Q$
(more precisely, a simple closed curve homotopic to $\gamma$ is contained in $\overline{\mathbb{C}}\backslash Q$).
And for any $\varepsilon>0$, there is a $\delta>0$ such that
\[
l_{\mu,E}(\gamma)>(1-\varepsilon)l_{\mu,Q}(\gamma)
\]
provided that $l_{\mu,E}(\gamma)<\delta$.
\end{lem}

\begin{proof}
Let $\gamma \subset RE_{\mu}$ be a non-peripheral simple closed geodesic. From Theorem~\ref{collar},
we have an annulus $A \subset RE_{\mu}$ as a collar such that $\gamma$ is the core curve of $A$
and
\begin{equation}\label{ine-1}
\frac{\pi}{2 l_{\mu, E} (\gamma)}-1<\mod(A) <
\frac{\pi}{2 l_{\mu, E} (\gamma)}.
\end{equation}

We may assume that $A$ separates $0$ and $\infty$.
Let $S_{1}$ and $S_{2}$ be the two components of $RE_{\mu}\setminus A$
such that $0 \in S_{1}$ and $\infty \in S_{2}$.  Let
$$
r = \max\{|z|\: \big{|}\: z \in  S_{1}\} \quad \hbox{and} \quad  R = \min\{ |z|\:\big{|}\: z \in
 S_{2}\}.
$$
By Lemma~\ref{mod}, when $l_{\mu, E} (\gamma)$ is small, $R/r$ is large. Consider the round annulus
$$
H = \{z\:\big{|}\: r < |z| < R\}.
$$
It follows that $H \subset A$ and that the core curve of $H$ is in
the same homotopic class as $\gamma$. By Lemma~\ref{mod} and
(\ref{ine-1}), it follows that there is a uniform constant $0< C <
\infty$ such that
\begin{equation}\label{mod-ine-2}
\mod(H) \ge \mod(A) - C
\end{equation}
holds provided that $l_{\mu, E}(\gamma)$ is small. Note that every
pair $\{p_{0}^{m}, p_{*}^{m}\}$ considered as points in $RE_{\mu}$ is contained either
in $\{z\:\big{|}\: |z| < r\}$ or in $\{z\:\big{|}\: |z| > R\}$.
Since $\mu$ is $0$ on $\overline{D_{m}} \cup A_{m}$ and $p_{0}^{m}, p_{*}^{m} \subset D_{m}$ considered as points and a domain in $RE_{\mu}$,
it follows from Koebe's distortion theorem
that there is an $1 < K < \infty$, which depends only on $\{D_{i}\}$
and $\{A_{i}\}$, such that every $\overline{D_{m}}$ considered as a domain in $RE_{\mu}$ is
contained either in $\{z\:\big{|}\: |z| < Kr\}$ or in
$\{z\:\big{|}\: |z| > R/K\}$. By (\ref{ine-1}) and
(\ref{mod-ine-2}), we have
$$
R/K > Kr
$$
provided that $l_{\mu, E}(\gamma)$ is small enough. All of these implies  that the annulus
$$
H_{K} = \{z\:\big{|}\: Kr < |z| < R/K\}
$$
is contained in $RQ_{\mu}$ considered as a sub-Riemann surface of $RE_{\mu}$ provided
that $l_{\mu, E}(\gamma)$ is small enough.

Now the first assertion of the lemma follows if we can show that
$$
\gamma \subset H_{K}
$$
provided that $l_{\mu, E}$ is small enough.
Suppose this were not true. Then there are two
cases. In the first case, there exist two points $z$ and $z'$ such
that
\begin{itemize}
\item[1.] $z \in S_{2}$ with $|z| = R$,
\item[2.] $|z'| = R/K$,
\item[3.] $\gamma$ separates $\{0, z'\}$ and $\{z, \infty\}$.
\end{itemize}
In the second case,  there exist two points $z$ and $z'$ such that
\begin{itemize}
\item[1.] $|z| = Kr$,
\item[2.] $z' \in S_{1}$ and $|z'| = r$.
\item[3.] $\gamma$ separates $\{0, z'\}$ and $\{z, \infty\}$.
\end{itemize}
Suppose we are in the first case. Note that the curve
$\gamma$ separates $A$ into two sub-annuli such that the
modulus of each of them is equal to $\mod(A)/2$. But on the
other hand,  the outer one separates $\{0, z'\}$ and $\{z,
\infty]\}$, and thus by Lemma~\ref{mod}, its modulus has an
upper bound depending only on $M$. By (\ref{ine-1}) this is
impossible when $l_{\mu, E} (\gamma)$ is small enough. The same
argument can be used to get a contradiction in the second case. This
proves the first assertion of the Lemma.

Now  let us prove the second assertion.  Let $l$ denote the
hyperbolic length of the core curve of $H_{K}$ with respect to the
hyperbolic metric of $H_{K}$. Since $H_{K} \subset RQ_{\mu}$ when $l_{\mu, E} (\gamma)$
is small enough, it follows that $l > l_{\mu, Q}$. Thus we have
$$
\mod(H_{K}) = \frac{\pi}{2l}  < \frac{\pi}{2 l_{\mu,
Q}}.
$$
From (\ref{ine-1}) and (\ref{mod-ine-2}), there is a constant $0<
C' < \infty$ such that
$$
\mod(H_{K}) \ge \frac{\pi}{2 l_{\mu, E} (\gamma)} - C'
$$
holds provided that $l_{\mu, E} (\gamma)$ is small enough.
Thus we have
$$
\frac{\pi}{2 l_{\mu,Q}}\le \frac{\pi}{2
l_{\mu, E}(\gamma)} \le \frac{\pi}{2 l_{\mu,Q}} + C'.
$$
The second assertion follows.
\end{proof}

\medskip
\begin{pro}~\label{cp}
Suppose $\mu$ is a Beltrami coefficient $\mu$ on $\overline{\mathbb{C}}$ with $\mu=0$ on $Q\cup \cup_{m=1}^{M} A_m$.
Suppose $\beta\subset \overline{\mathbb{C}}\backslash Q$ is a non-peripheral simple closed curve.
When $l_{\mu, Q}(\beta)$ sufficient small,
we have
\[
l_{\mu,Q}(\beta)\geq \frac{l_{\mu,E}(\beta)}{1+\frac{2c}{\pi}l_{\mu,E}(\beta)}
\]
where $c>1$ is a constant.
\end{pro}

\begin{proof}
Suppose $\beta$ is a closed geodesic in $RQ_{\mu}$. Since $l_{\mu, Q}(\beta) >l_{\mu, E} (\beta)$, if
$l_{\mu, Q}(\beta)$ is sufficiently small, then $l_{\mu, E} (\beta)$ is sufficiently small, Lemma~\ref{zj} applies.

For $l_{\mu, Q}(\beta)$ sufficiently small, there exists a collar $A_\beta$ of $\beta$ (refer to the proof of Lemma~\ref{zj})
such that
\[
\frac{\pi}{2}\frac{1}{l_{\mu, Q}(\beta)}-1<\mod(A_\beta)<\frac{\pi}{2}\frac{1}{l_{\mu, Q}(\beta)}
\]
and $\mod(A_\beta)$ sufficiently large. Suppose $A_\beta$ separates $0$ and $\infty$. Similar to the proof of Lemma~\ref{zj},
we get a round annulus $H_K$ such that $H_K\subset A_\beta$ and the core curve of $H_K$ is homotopic to $\beta$ in $RE_{\mu}$.
Since $l_{\mu, Q}(\beta)$ sufficiently small, similar to the proof of Lemma~\ref{zj}, there exists a constant $0<c'<\infty$ such that
\[
\mod(H_K)\geq \mod(A_\beta)-c'\geq \frac{\pi}{2}\frac{1}{l_{\mu, Q}(\beta)}-1-c'
=\frac{\pi}{2}\frac{1}{l_{\mu, Q}(\beta)}-c
\]
where $c=1+c'>1$. Since, from the proof of Lemma~\ref{zj},
\[
\mod(H_K)\leq \frac{\pi}{2}\frac{1}{l_{\mu, E}(\beta)},
\]
so
\[
\frac{\pi}{2}\frac{1}{l_{\mu, Q}(\beta)}-c\leq \frac{\pi}{2}\frac{1}{l_{\mu, E}(\beta)},
\]
i.e.
\[
\frac{1}{l_{\mu, Q}(\beta)}\leq \frac{1+\frac{2c}{\pi}l_{\mu, E}(\beta)}{l_{\mu, E}(\beta)}.
\]
This implies the proposition.
\end{proof}

Recall the Teichm\"uller space $T_f$ for a sub-hyperbolic semi-rational branched covering $f$ is defined
the Teichm\"uller space modeled on $(\overline{\mathbb{C}}\backslash Q,X)$.
Similarly, we can define the Teichm\"uller space is the Teichm\"uller space modeled
on $(\overline{\mathbb{C}}\backslash Q,X)$ for sub-hyperbolic semi-rational {\em type} branched covering.
Based on the definition, the points of Teichm\"uller space can be represented by $[\mu]$ where $\mu$ is the Beltrami
coefficient of the normalized quasiconformal mapping $w^{\mu}$ and $[\mu]$ is the equivalent class of $\mu$.
Let $\mu$ be a Beltrami coefficient defined on $(\overline{\mathbb{C}}\backslash Q,X)$.
Extended it to the Beltrami coefficient $Ext(\mu)(z)$ on $\overline{\mathbb{C}}$ by setting
\[
Ext(\mu)(z)=
\begin{cases}
\mu(z)\ \ for\  z\in (\overline{\mathbb{C}}\backslash Q,X),\\
0\ \ for\  otherwise.
\end{cases}
\]
If no confusion, we will simply use $\mu$ to denote $Ext(\mu)$ or $\mu$.

Remember that we started from a sub-hyperbolic semi-rational branched covering $f$
which is not CLH-equivalent to a rational map. For a given point $\tau_0$ in Teichm\"uller space, it can be denoted by $[\mu_0]$, that is, the Beltrami coefficient is $\mu_0$ on $(\overline{\mathbb{C}}\backslash Q,X)$ while $0$ on other points. $\tau_n$ is defined by $\tau_n=\sigma_f^n(\tau_0)=[(f^*)^n(\mu_0)]$.
 Decomposing the Riemann surface $(\overline{\mathbb{C}}\backslash Q,X)$
along canonical Thurston obstruction $\Gamma_{c}$, we got the thick-thin decomposition.
The exact same assumption as that in the decomposition section and in the extension section,
we suppose $P_i^0$ is periodic with period $k$. Let $\widetilde{f}$ be the map defined in (\ref{newmap}) and
\[
\widetilde{P}_0=P_i^0\bigcup (\bigcup_{i=1}^p D(\gamma_i))
\]
where $\{D(\gamma_i)\}$ are the disks in the extension section. Here $\widetilde{P}_0$ can be though
as the Riemann sphere $\overline{\mathbb{C}}$ as our starting space.

Let $\tau_0=[\mu_0]$ be a given complex structure on $\overline{\mathbb{C}}\backslash Q$
which is extended to a complex structure on $\overline{\mathbb{C}}$ by setting $\mu_{0}=0$ on $Q$.
Define $\widetilde{\tau}_0=[\widetilde{\mu}_0]$ on $\overline{\mathbb{C}}$ as
\[
\widetilde{\mu}_0(z)=
\begin{cases}
\mu_0(z),& z\in P_i^0;\\
0, & z\in \cup_{i=1}^p D(\gamma_i).
\end{cases}
\]
Using the induced pullback maps $\sigma_{f^k}$ and
$\sigma_{\widetilde{f}}$ on the corresponding Teichm\"uller spaces,
we get two sequences of points on the corresponding Teichm\"uller spaces:
$$
\tau_n=\sigma^n_{f^k}(\tau_0) =[\mu_{n}] \quad \hbox{and}\quad \widetilde{\tau}_n=\sigma_{\widetilde{f}}^n (\widetilde{\tau}_0) =[\widetilde{\mu}_{n}].
$$

Let $\tau_n'''=[\mu_n''']$ be the complex structure on
$\overline{\mathbb{C}}$ as
\[
\mu_n'''(z)=
\begin{cases}
\mu_n(z), & z\in P_i^0;\\
0, & z\in \cup_{i=1}^p D(\gamma_i).
\end{cases}
\]
We would like to point out that $\tau_n'''\neq \widetilde{\tau}_{n}$.

Let $\tau_n''=[\mu_n'']$ be the complex structure on
$\overline{\mathbb{C}}\backslash E$, where $E=P_1\cup
(\cup_{m=1}^M\{p_0^m,p_*^m\})$ defined in (\ref{newset}), by setting
\[
\mu_n''(z)=\mu_n'''(z),\ \ \ z\in \overline{\mathbb{C}}\backslash
P_1^*
\]
where $P_1^*=P_f|_{P_i^0}\cup (\cup_{m=1}^M \{p_{0}^{m},p_*^m\})$.
Let $\tau_n'=[\mu_n']$ be the complex structure on
$\overline{\mathbb{C}}\backslash E$ by setting
\[
\mu_n'(z)=\widetilde{\mu}_n(z),\ \ \ z\in
\overline{\mathbb{C}}\backslash P_1^*.
\]
Therefore, we got five different complex structures
$$
\tau_n=[\mu_n], \;\; \widetilde{\tau}_n=[\widetilde{\mu}_n],\;\; \tau_n'=[\mu_n'],\;\;
\tau_n''=[\mu_n''],\;\; \hbox{and}\ \tau_n'''=[\mu_n''']
$$
on two different surfaces.
We now list relationships between those complex structures as a proposition which is helpful for a better understanding of the
procedure used in the proof of our main theorem.

\medskip
\begin{pro}\label{relationship} Suppose $\tau_n, \widetilde{\tau}_n, \tau_n', \tau_n'', \tau_n'''$ are the above complex structures. Then
\begin{itemize}
\item[(1)] $\mu_n'(z)=\widetilde{\mu}_{n}(z),\ \ z\in \overline{\mathbb{C}}\backslash P_1^*$;
\item[(2)] $\mu_n''(z)=\mu_n'(z), \;  z\in \overline{\mathbb{C}}\backslash (E\cup (\cup_{l=1}^n \widetilde{f}^{-l}(\cup_{i=1}^p D(\gamma_i) )))$,\\
$\mu_n''=\widetilde{\mu}_{n}, \ \ z\in \overline{\mathbb{C}}\backslash (P_1^* \cup (\cup_{l=1}^n \widetilde{f}^{-l}(\cup_{i=1}^p D(\gamma_i))))$;
\item[(3)] $\mu_n'''(z)=\mu_n''(z),\ \ z\in \overline{\mathbb{C}}\backslash P_1^*$,\\
$\mu_n'''(z)=\mu_n'(z),\ \ z\in P_i^0\backslash (P_1^* \cup (\cup_{l=1}^n \widetilde{f}^{-l}(\cup_{i=1}^p D(\gamma_i) )))$,\\
$\mu_n'''(z)=\widetilde{\mu}_n(z),\ \ z\in \overline{\mathbb{C}}\backslash (P_{1}^*\cup (\cup_{l=1}^n \widetilde{f}^{-l}(\cup_{i=1}^p D(\gamma_i) )))$;
\item[(4)] $\mu_n=\mu_n''',\ \ z\in P_i^0 \backslash P_f|_{P_i^0}$,\\
$\mu_n=\mu_n'',\ \ z\in P_i^0\backslash P_1^*$,\\
$\mu_n=\mu_n',\ \ z\in P_i^0\backslash (P_1^* \cup (\cup_{l=1}^n \widetilde{f}^{-l}(\cup_{i=1}^p D(\gamma_i) ))),$\\
$\mu_n=\widetilde{\mu}_n,\ \ n\geq 1,\ \ z\in P_i^0\backslash (P_f|_{P_i^0} \cup
(\cup_{l=1}^n \widetilde{f}^{-l}(\cup_{i=1}^p D(\gamma_i) ))$,\\
$\mu_0=\widetilde{\mu}_0,\ \ z\in P_i^0 \backslash P_f|_{P_i^0}$.
\end{itemize}
\end{pro}

The following theorem  is proved in~\cite{DH} (see also~\cite{Se}). This theorem is for a post-critically finite branched covering. For a sub-hyperbolic semi-rational branched covering, since $\sharp P_f=\infty$, this theorem can not be used directly. We will use the idea in the theorem in the proof of our main theorem.

\medskip
\begin{thm}[\cite{DH}]~\label{cov}
Suppose $f:\mathbb{S}^2\rightarrow \mathbb{S}^2$ is a post-critically finite branched covering map. Then there exist two covering maps  $\pi_1:T_f\rightarrow M_f'$, $\pi_2:M_f'\rightarrow M_f$ and a map $\widetilde{\sigma}_f: M_f'\rightarrow M_f$ such that $\pi_2$ is finite and
$\pi\circ \sigma_f=\widetilde{\sigma}_f\circ\pi_1$ where $\pi=\pi_2\circ\pi_1$, $M_f$ is the moduli space of $\mathbb{S}^2\backslash P_f$ and $M_f'$ is the intermediate moduli space.
\end{thm}

%
%
%



\begin{proof}[Proof of Theorem~\ref{mainth} (the Main Theorem)]
If $\widetilde{f}$ is a post-critically finite type branched covering
with hyperbolic orbifold, by using a similar method to that in~\cite{Se,BY},
we can prove that $\widetilde{f}$ has no Thurston obstruction
and thus combinatorially equivalent to a rational map.
Our main effort in this paper is to prove the theorem when
$\widetilde{f}$ is a sub-hyperbolic semi-rational type branched covering.

Now suppose $\tilde{f}$ is a sub-hyperbolic semi-rational branched covering.
First we prove that $\{\tau_{n}\}$ has no any convergent subsequence if $\widetilde{f}$ is
not CLH-equivalent to a rational map. We prove it by contradiction as follows.

Suppose there exists a sub-sequence $\{\tau_{n_k}\}$ of
$\{\tau_{n}\}$ convergence. Since $\widetilde{f}$ is holomorphic in
holomorphic disks, $\{\tau_{n_k}''\}$ is also convergent. Suppose
$\tau_{n_k}''\rightarrow \tau_{0,*}''$. Then
$\{\tau_{n_k+1}\}=\{\sigma_{f}(\tau_{n_k})\}$ is
convergent since $\sigma_{f}$ is continuous.
By the same reason, we get $\{\tau_{n_k+1}''\}$ convergent.
Suppose $\tau_{n_k+1}''\rightarrow \tau_{1,*}''$.
$\tau_{1,*}''=\sigma_{\widetilde{f}}(\tau_{0,*}'')$. Continue the
above procedure, we get the convergent sequence $\{\tau_{n_k+i}''\}$
which converge to $\tau_{i,*}''\ (i=0,1,\cdots)$.
Any non-peripheral simple closed curve $\gamma$ in $P_{i}^{0}$ has common
lower bound hyperbolic length $K>0$, that is, for any $n_k$ and any
$i$, $l_{\tau_{n_k+i}} (\gamma)\geq K$ (see Lemma~\ref{bigconst}).
By Proposition~\ref{samebig}, there exists a constant $\tilde{K}=\tilde{K}(K)>0$
such that $l_{\tau_{n_k+i}''}(\gamma)\geq \tilde{K}$.
Since $\tau_{n_k+i}''\rightarrow \tau_{i,*}''\ (n_k\rightarrow \infty)$,
so $l_{\tau_{i,*}''}(\gamma)\geq \tilde{K}$. We use $\pi^*$ denote the mapping
which replace the points of $P_{\widetilde{f}}$ by $E$.
Write $\pi^*(\widetilde{\tau}_i)=\sigma_{\widetilde{f}}^i(\tau_{0,*}'')$.
Since $\widetilde{f}$ is not CLH-equivalent to a rational map, there exists
at least one non-peripheral simple closed curve $\beta$ such that
$l_{\widetilde{\tau}_i}(\beta)\rightarrow 0$ as $i\rightarrow\infty$ from Theorem~\ref{cjthm} .
By Proposition~\ref{cp}, we know $l_{\tau_{i,*}''}(\beta)\rightarrow 0$.
This is a contradiction.

Suppose $\{\tau_{n_k}\}$ is a subsequence which is not convergent.
By Proposition~\ref{relationship}, $\{\tau_{n_k}''\}$ does not convergent.
Let $\pi''$ be the mapping which send $\tau_n$ to $\tau_n''$.
Any non-peripheral simple closed curve $\gamma$ in $P_{i}^{0}$ has common
lower bound hyperbolic length $K>0$, that is, for any $n_k$ and any
$i$, $l_{\tau_{n_k+i}} (\gamma)\geq K$ (see Lemma~\ref{bigconst}).
By Proposition~\ref{samebig}, there exists a constant $\tilde{K}=\tilde{K}(K)>0$
such that $l_{\tau_{n_k+i}''}(\gamma)\geq \tilde{K}$
for any non-peripheral simple closed curve $\gamma$ in
$\overline{\mathbb{C}}\backslash (P_{1}\cup (\cup_{m=1}^M \{p_{0}^{m},p_*^m\}))$.

Let $E=P_{1}\cup (\cup_{m=1}^M \{p_{0}^{m},p_*^m\})$ be the set defined in (\ref{newset}).
Consider the Teichm\"uller space $T=T(\overline{\mathbb{C}}\backslash E)$.
Let $M=M (\overline{\mathbb{C}}\backslash E)$ be the moduli space.
Let $\pi: T\to M$ be the covering map. For the subsequence $\{\tau_{n_k}\}$,
then $\{ \pi''(\tau_{n_k}))\}$ is a sequence in $T$. Similar to Theorem~\ref{cov}, we have an intermediate
moduli space $M'$ and $\pi_{1}: T\to M'$ and $\pi_{2}: M'\to M$ satisfying the condition in Theorem~\ref{cov}.
Then $\{m_{n_k}\}=\{\pi_1 (\pi''(\tau_{n_k})) \}$ is a sequence in $M'$.
From the previous paragraph and the Mumford compactness theorem, $\{m_{n_k}\}$ converges to some $\widehat{m}_0$ in $M'$.
By selecting a sub-subsequence, which we still denote as $\{m_{n_k}\}$, such that
$m_{n_k+i}\rightarrow \widehat{m}_i\ (i=0,1,\cdots)$ in $M'$. The set $\{\widehat{m}_i\}\ (i=0,1,\cdots)$ lies in a compact subset in $M'$.

Suppose $d(m_{n_k},\widehat{m}_0)<\varepsilon$ and $d(m_{n_k+1},\widehat{m}_1)<\varepsilon$.
Then there exists a point $\widehat{\tau}_k$ such that $d(\tau_{n_k}'',\widehat{\tau}_k)=d(m_{n_k},\widehat{m}_0)<\varepsilon$
and $\pi_1(\tau_{n_k}'')=m_{n_k}$, $\pi_1(\widehat{\tau}_k)=\widehat{m}_0$. So
\[
d(\widehat{m}_1, \pi_1(\pi^*(\sigma_{\widetilde{f}}(\widehat{\tau}_k))))\leq d(\widehat{m}_1,m_{n_k+1})+
d(m_{n_k+1},\pi_1(\pi^*(\sigma_{\widetilde{f}}(\widehat{\tau}_k))))
\]
\[
\leq \varepsilon+d(\tau''_{n_k+1},\pi^*(\sigma_{\widetilde{f}}(\widehat{\tau}_k)))
\leq \varepsilon+q d(\tau_{n_k}'',\widehat{\tau}_k)
\leq (q+1)\varepsilon
\]
where $q\in [0,1)$ is a constant only depending on the compact set which contains all $\{\widehat{m}_i\}\ (i=0,1,\cdots)$ in $M'$.

Note that $\pi_{2}: M'\to M$ is a cover of finite degree and $\pi: T\to M$ is a cover of infinite degree.
Since $\pi(\pi^*(\sigma_{\widetilde{f}}(\widehat{\tau}_k)))=\pi_2(\widehat{m}_1)$,
we have $\pi_1(\pi^*(\sigma_{\widetilde{f}}(\widehat{\tau}_k)))$ lies in the fiber $\pi_2^{-1} (\pi_2(\widehat{m}_1))$.
Since $\pi_2$ is a cover of finite degree, set $c_i$ to be the minimal distance between any two different points
in the fiber $\pi^{-1}_2 (\pi_2(\widehat{m}_i))\ (i=0,1,\cdots)$. From the above, we can conclude $\pi_1(\pi^*(\sigma^i_{\widetilde{f}}(\widehat{\tau}_k)))=\widehat{m}_i$ if $d(m_{n_k+t},\widehat{m}_i)<\varepsilon\ (t=0,1,\cdots,i)$ and $\varepsilon \leq \frac{1}{q+1}\min\{c_1,c_2,\cdots,c_i\}$.

For any given $i$,
\[
d(\widehat{m}_i,\widehat{m}_{i+1})= d(\pi^*(\sigma^i_{\widetilde{f}}(\widehat{\tau}_k)),\pi^*(\sigma^{i+1}_{\widetilde{f}}(\widehat{\tau}_k)))
\]
\[
\leq q^i d(\widehat{\tau}_k,\widehat{\tau}_{k+1})=q^i d(\widehat{m}_0,\widehat{m}_1).
\]
Since the above estimate does not depend on the selection of $\widehat{\tau}_k$ as long as $\{\widehat{m}_i\}\ (i=0,1,\cdots)$ lies in the given compact subset of $M'$. So for any $i$
\[
d(\widehat{m}_i,\widehat{m}_{i+1})\leq q^i d(\widehat{\tau}_k,\widehat{\tau}_{k+1})=q^id(\widehat{m}_0,\widehat{m}_{1}).
\]
Hence $\{\widehat{m}_i\}$ converges to some $\widehat{m}$ which lies in the given compact subset of $M'$.
So there exists some $\tau\in T$ such that $\pi_1(\tau'')=\widehat{m}$. Hence $\tau$ lies in the fiber $\pi_1^{-1}(\widehat{m})$ and $\sigma_{\widetilde{f}}(\tau)$ lies in the fiber $\pi^{-1} (\pi_2(\widehat{\pi}))$. Since
\[
d(\tau'',\pi^*\circ\sigma_{\widetilde{f}}(\tau''))
\]
\[
\leq d(\tau'',\pi^*\circ\sigma^t_{\widetilde{f}}(\widehat{\tau}_k))
+d(\pi^*\circ\sigma^t_{\widetilde{f}}(\widehat{\tau}_k),\pi^*\circ\sigma^{t+1}_{\widetilde{f}}(\widehat{\tau}_k))
\]
\[
+d(\pi^*\circ\sigma^{t+1}_{\widetilde{f}}(\widehat{\tau}_k),\pi^*\circ\sigma_{\widetilde{f}}(\tau''))
\]
\[
\leq 2\varepsilon+q^t d(\widehat{m}_0,\widehat{m}_1),
\]
we can select $\varepsilon$ sufficiently small and $t$ sufficiently large such that the lower bound of Teichm\"uller space between different elements in the fiber $\pi^{-1}\circ \pi_2(\widehat{m})$ larger than $2\varepsilon+q^t d(\widehat{m}_0,\widehat{m}_1)$.
So $\tau=\pi^*\circ \sigma_{\widetilde{f}}(\tau)$. Hence $(\pi^*)^{-1}(\tau)=\sigma_{\widetilde{f}}\circ(\pi^*)^{-1}(\tau)$ which implies $\sigma_{\widetilde{f}}$ has a fixed point in the Teichm\"uller space $T_f$. Thus $\widetilde{f}$ is equivalent to a rational map.
It completes the proof.
\end{proof}


\begin{thebibliography}{ABC}

\bibitem[Ah1]{Ah1} L. V. Ahlfors, {\em Lectures on Quasiconformal Mappings}. Van Nostrand Mathematical studies, Vol. {\bf 10},
D. Van Nostrand Co. Inc., Toronto-New York-London, 1966.

\bibitem[Ah2]{Ah2} L. V. Ahlfors, {\em Conformal invariants}. McGraw-Hill, New York, 1973.

\bibitem[BY]{BY} S. Bonnot and M. Yampolsky, Geometrization of post-critically finite branched coverings. ArXiv: 1010.1690 V1

\bibitem[ChJ]{ChJ} T. Chen and Y. Jiang, Canonical characterization of sub-hyperbolic semi-rational branched coverings. ArXiv: 1101.2285 V2

\bibitem[ChJ1]{ChJ1} T. Chen and Y. Jiang, Bounded geometry and the characterization of rational maps. Preprint.

\bibitem[CJ]{CJ} G. Cui and Y. Jiang, Geometrically finite and semi-rational branched coverings of the two sphere.
{\em Transactions of American Mathematical Society}, Vol. {\bf 363} (2011), 2701-2714.

\bibitem[CJS]{CJS} G. Cui, Y. Jiang and D. Sullivan, On geometrically finite branched covering maps I: locally combinatorial attracting and II: Realization of rational maps. {\em Complex Dynamics and Related Topics, New Studies in Advanced Mathematics}, 2004, The International Press, 1-14 \& 15-29.

\bibitem[CT]{CT} G. Cui and L. Tan, A characterization of hyperbolic rational maps. {\em Invent Math.}, Vol. {\bf 183} (2011), 451-516.

\bibitem[DH]{DH} A. Douady and J.H. Hubbard, A proof of Thurston's topological characterization of rational functions. {\em Acta Math.}, Vol. {\bf 171} (1993), 263-297.

\bibitem[Ha]{Ha} P. Haissinsky, Deformation loaclisee de surfaces de Riemann. {\em Publ. Math.},  Vol. {\bf 49} (2005), 249-255.

\bibitem[J]{J} Y. Jiang, A framework toward understanding of the characterization of holomorphic maps. To appear in the Proceedings of Milnor 80th birthday conference.

\bibitem[Mi]{Mi} Y. Minsky, Extremal length estimates and product regions in Teichm\"uller space. {\em Duke Math. J.}, Vol. {\bf 83} (1996), 249-286.

\bibitem[Pi1]{Pi1} K. M. Pilgrim, Canonical Thurston obstruction, {\em Advances in Mathematics}, Vol. {\bf 158} (2001), 154-168.

\bibitem[Pi2]{Pi2} K. M. Pilgrim, Combinations of complex dynamical systems, {\em Lecture Notes in Mathematics}, Vol. 1827, Springer, 2003.

\bibitem[Se]{Se} N. Selinger, Thurston's pullback map on the argmented Teichm\"uller space and applications. Inventiones Mathematicae
DOI: 10.1007/s00222-011-0362-3.

\bibitem[ZJ]{ZJ} G. Zhang and Y. Jiang, Combinatorial characterization of sub-hyperbolic rational map. {\em Advances in Mathematics}, Vol. {\bf 221} (2009), 1990-2018. MR2522834 (2010h:37095)

\end{thebibliography}
\end{document}